\documentclass[11pt, a4paper]{amsart}
\usepackage{amssymb, array, amsmath, amscd, pdfpages, enumerate, amsthm, setspace, hyperref,mathtools}

\usepackage[utf8]{inputenc}
\usepackage{amssymb}
\usepackage{bm}
\usepackage{graphicx}

\newcommand{\yref}{y_{\mbox{\scriptsize\textit{ref}}}}
\newcommand{\wdist}{w_{\mbox{\scriptsize\textit{dist}}}}

\newcommand{\wdistj}[1][1]{w_{\mbox{\scriptsize\textit{dist}},#1}}

\newcommand{\mc}[1]{\mathcal{#1}}
\newcommand{\mf}[1]{\mathfrak{#1}}

\newcommand{\citel}[2]{\cite[#2]{#1}}

\newcommand{\yrk}[1][k]{y_r^{#1}}
\newcommand{\wdkone}[1][k]{w_{1#1}}
\newcommand{\wdktwo}[1][k]{w_{2#1}}
\newcommand{\wdkthree}[1][k]{w_{3#1}}

\renewcommand{\P}{P}

\newtheorem{theorem}{Theorem}[section]
\newtheorem{proposition}[theorem]{Proposition}
\newtheorem{lemma}[theorem]{Lemma}

\newtheorem{assumption}[theorem]{Assumption}

\theoremstyle{definition}
\newtheorem{definition}[theorem]{Definition}

\newcommand{\pmat}[1]{\begin{bmatrix}#1 \end{bmatrix}}
\newcommand{\pmatsmall}[1]{\begin{bsmallmatrix}#1\end{bsmallmatrix}}
\newcommand{\bmat}[1]{\begin{bmatrix}#1 \end{bmatrix}}

\newcommand{\eps}{\varepsilon}

\DeclareMathOperator{\re}{Re}

\DeclareMathOperator{\diag}{diag}
\DeclareMathOperator{\blkdiag}{blockdiag}

\newcommand*{\C}{{\mathbb{C}}}     
\newcommand*{\R}{{\mathbb{R}}}

\newcommand*{\Lin}{{\mathcal{L}}}   
\newcommand*{\Dom}{D}   
\newcommand{\ran}{{\mathcal{R}}}   
\renewcommand{\ker}{{\mathcal{N}}}

\newcommand*{\abs} [1]{\lvert#1\rvert}
\newcommand*{\norm}[1]{\lVert#1\rVert}
\newcommand*{\set} [1]{\{#1\}}
\newcommand*{\setm}[2]{\{\,#1\mid#2\,\}}   
\newcommand*{\iprod}[2]{\langle#1,#2\rangle}

\newcommand*{\Lp}[1][p]{L^{#1}}

\newcommand{\Abs}[2][default]{\ifthenelse{\equal{#1}{default}}{\left\lvert#2\right\rvert}{\ldelim{#1}{\lvert}#2\rdelim{#1}{\rvert}}}
\newcommand{\Norm}[2][default]{\ifthenelse{\equal{#1}{default}}{\left\lVert#2\right\rVert}{\ldelim{#1}{\lVert}#2\rdelim{#1}{\rVert}}}

\newcommand*{\Iprod}[3][default]{\ifthenelse{\equal{#1}{default}}{\left\langle#2,#3\right\rangle}{\ldelim{#1}{\langle}#2,#3\rdelim{#1}{\rangle}}}
\newcommand*{\Dualpair}[3][default]{\ifthenelse{\equal{#1}{default}}{\left\langle#2,#3\right\rangle}{\ldelim{#1}{\langle}#2,#3\rdelim{#1}{\rangle}}}

\newcommand*{\List}[2][1]{\set{#1,\ldots,#2}}

\newcommand{\eq}[1]{\begin{align*}#1\end{align*}}
\newcommand{\eqn}[1]{\begin{align}#1\end{align}}

\newcommand{\gs}{\sigma}
\newcommand{\ga}{\alpha}

\newcommand{\gd}{\delta}
\newcommand{\gl}{\lambda}
\newcommand{\gw}{\omega}

\newcommand{\ieq}[1]{$#1$}

\newcommand{\inv}{^{-1}}

\newcommand*{\ddb}[2][1]{\ifthenelse{\equal{#1}{1}}{\frac{d}{d#2}}{\frac{d^{#1}}{d#2^{#1}}}}
\newcommand*{\pd}[3][1]{\ifthenelse{\equal{#1}{1}}{\frac{\partial{#2}}{\partial{#3}}}{\frac{\partial^{#1}{#2}}{\partial#3^{#1}}}}

\newcommand*{\keyterm}[1]{\emph{#1}}
\newtheorem*{RORP}{The Robust Output Regulation Problem}

\begin{document}

  \title[A Lyapunov Approach to Robust Regulation of PHS]{A Lyapunov Approach to Robust Regulation of Distributed Port--Hamiltonian Systems}

\thispagestyle{plain}

\author[L. Paunonen]{Lassi Paunonen}
\address[L. Paunonen]{Mathematics, Faculty of Information Technology and Communication Sciences, Tampere University, PO.\ Box 692, 33101 Tampere, Finland}
\email{lassi.paunonen@tuni.fi}

\author[Y. Le Gorrec]{Yann Le Gorrec}
\address[Y. Le Gorrec]{FEMTO-ST Institute, AS2M department, Universit\'{e} de Franche-Comt\'{e}, Besan\c{c}on, France}
\email{legorrec@femto-st.fr}

\author[H. Ram\'{i}rez]{H\'{e}ctor Ram\'{i}rez}
\address[H. Ram\'{i}rez]{Universidad Tecnica Federico Santa Maria, Valparaiso, Chile}
\email{hector.ramireze@usm.cl}

\thanks{This work was supported by the Academy of Finland Grants number 298182 and 310489 held by L. Paunonen, the Agence Nationale de la Recherche/Deutsche Forschungsgemeinschaft (ANR-DFG) project INFIDHEM, ID ANR-16-CE92-0028, and the EUROPEAN COMMISSION through H2020-ITN-2017-765579 --- ConFlex,
and AC3E basal project FB0008. }

\begin{abstract}
This paper studies robust output tracking and disturbance rejection for boundary controlled infinite-dimensional port--Hamiltonian systems including second order models such as the Euler--Bernoulli beam. The control design is achieved using the internal model principle and the stability analysis using a Lyapunov approach. Contrary to existing works on the same topic no assumption is made on the external well-posedness of the considered class of PDEs. The results are applied to robust tracking of a piezo actuated tube used in atomic force imaging.
\end{abstract}


\maketitle

\section{Introduction}

We consider robust output regulation for a class of linear partial differential equations (PDEs) with boundary control and observation, 
namely,
\keyterm{port-Hamiltonian systems} (PHS)~\cite{LeGZwa05,JacZwa12book} 
\begin{subequations}
  \label{eq:PHSintro}
  \eqn{
    \label{eq:PHSintro1}
    \frac{\partial x}{\partial t}(z,t)
    &=P_2\frac{\partial^2 }{\partial z^2}\big({\mathcal H}(z)x(z,t)\big)+P_1\frac{\partial }{\partial z}\big({\mathcal H}(z)x(z,t)\big)
    \\[1ex]  
    \label{eq:PHSintro2}
    &\quad +(P_0-G_0)\left({\mathcal H}(z)x(z,t)\right)
    + B_d(z) \wdistj[1](t), 
    \\[1ex]
      \label{eq:PHSintroBCs}
    W_1
    \pmat{ f_\partial(t) \\ e_\partial(t) }
    &= u(t) + \wdistj[2](t),
    \quad
W_2 \pmat{ f_\partial(t) \\ e_\partial(t) }
    = \wdistj[3](t)
    \\[1ex]
     y(t)&=\tilde{W}
    \pmat{ f_\partial(t) \\ e_\partial(t) } 
  }
\end{subequations}
on a one-dimensional spatial domain $[a,b]$ (see Section~\ref{sec:PHSSummary} for detailed assumptions).
In robust regulation, the purpose of the control $u(t)\in\R^p$ is to achieve the asymptotic convergence of the output $y(t)\in \R^p$ of~\eqref{eq:PHSintro} to a given reference signal $\yref(t)$, i.e., $\norm{y(t)-\yref(t)}\to 0$ as $t\to \infty$,
despite external disturbance signals 
  $\wdist(t):=(\wdistj[1](t),\wdistj[2](t),\wdistj[3](t))$.
The signals $\yref(t)$ and $\wdist(t)$ are assumed to have the forms%
\begin{subequations}%
    \label{eq:yrefdist}
  \eqn{
    \label{eq:yrefdist1}
    \yref(t) &= a_0+\sum_{k=1}^q \left[a_k^1 \cos(\gw_k t) + a_k^2 \sin(\gw_k t)\right], \\
    \label{eq:yrefdist2}
    \wdist(t) &= b_0+\sum_{k=1}^q \left[b_k^1 \cos(\gw_k t) + b_k^2 \sin(\gw_k t)\right],
  }
\end{subequations}
for known frequencies $0=\gw_0<\gw_1<\cdots<\gw_q$ and unknown amplitudes $\set{a_k^1}_{k=0}^q,\set{a_k^2}_{k=1}^q\subset \R^p$, and $\set{b_k^1}_{k=0}^q,\set{b_k^2}_{k=1}^q\subset \R^{n_{d,1}+p+n_{d,3}}$.

Several recent articles have considered output regulation for individual linear PDEs, such as 1D heat equations~\cite{Deu15}, beam equations~\cite{JinGuo19} and wave equations~\cite{GuoZho18}.
In this paper we solve the control problem
 for a \keyterm{class} of boundary controlled 1D PDEs~\eqref{eq:PHSintro},
which 
covers many particular hyperbolic PDE systems such as boundary controlled wave equations, Schr\"odinger equations, Timoshenko and Euler--Bernoulli beam models with spatially varying physical parameters, and is used in modeling and control of flexible structures, heat exchangers, and chemical reactors.  
We focus here on \keyterm{impedance passive} PHS~\eqref{eq:PHSintro}, and solve the output regulation problem using 
a finite-dimensional 
dynamic error feedback controller 
\begin{subequations}
  \label{eq:controllerintro}
    \eqn{
      \dot{x}_c(t)&= J_c x_c(t) + \gd_c B_c (\yref(t)-y(t)), 
      \quad x_c(0)\in X_c
      \\
      \label{eq:controllerintroOutput}
      u(t)& = \gd_c B_c^\ast  x_c(t)  + D_c(\yref(t)-y(t))
    }
\end{subequations}
where $J_c$ is skew-symmetric, $B_c\in\R^{p\times n_c}$, and $D_c\in\R^{p\times p}$ satisfies $D_c\geq 0$. 
Finally, $\gd_c>0$ is a gain parameter.
In studying the class~\eqref{eq:PHSintro} of PDEs our aim is to design the controller~\eqref{eq:controllerintro} under assumptions that can be verified directly based on the properties of the original PDE~\eqref{eq:PHSintro} and the matrices $(P_0,P_1,P_2,G_0,W_1,W_2,\tilde{W})$, without the need to reformulate~\eqref{eq:PHSintro} as an abstract system.

Our results for the class~\eqref{eq:PHSintro} are based on the theoretical results on robust output regulation of \keyterm{abstract boundary control and observation systems}~\cite{CheMor03,TucWei09book} presented in this paper.
They extend the theory related to internal model based controllers for passive \keyterm{well-posed linear systems} and PHS in~\cite{RebWei03,HamPoh10,HamPoh11,HumPau18,HumKur19}, and they compose the main technical contributions of the paper. 
In particular, we 
introduce a new Lyapunov-type argument for the stability analysis of the closed-loop system consisting of the boundary control system and the controller (extending our earlier results in~\cite{PauLeGLHMNC18} for PHS with distributed control and observation).
In addition, the controller design is done without assuming well-posedness of the original control system (which was assumed in~\cite{RebWei03}) and the analysis is completed directly in the abstract boundary control system framework (whereas in~\cite{HumPau18,HumKur19} the boundary control inputs were first reformulated as distributed inputs using a state extension). 
The class~\eqref{eq:PHSintro} includes models which are not wellposed (in the sense of~\citel{Sta02}{Sec.~2}).
The stability analysis of the closed-loop system is also related to references~\cite{RamLeG14,MacCal18} studying the stability of coupled impedance passive systems in a different context {\it i.e.} when the infinite dimensional system is undamped and the controller strictly input passive.

The paper is organised as follows. In Section \ref{sec:PHSSummary} we define the considered class of boundary controlled PHS and state our main result for the PDEs~\eqref{eq:PHSintro} (these are proved later in Section~\ref{sec:RORPcontroller}).
In Sections~\ref{sec:BCS}--\ref{sec:RORPcontroller} we present our main results for
abstract boundary control systems. 
The results are applied in solving a concrete output regulation problem in Section~\ref{sec:example}.
The paper ends with some conclusions and perspectives.

\noindent\textbf{Notation.}
  If $X$ and $Y$ are Banach spaces and $A:X\rightarrow Y$ is a linear operator, we denote by $\Dom(A)$, $\ker(A)$ and $\ran(A)$ the domain, kernel and range of $A$, respectively. The space of bounded linear operators from $X$ to $Y$ is denoted by $\Lin(X,Y)$. If \mbox{$A:X\rightarrow X$,} then $\gs(A)$
  and $\rho(A)$ denote the spectrum
  and the \mbox{resolvent} set of $A$, respectively. For $\gl\in\rho(A)$ the resolvent operator is  \mbox{$R(\gl,A)=(\gl -A)^{-1}$}.  The inner product on a Hilbert space is denoted by $\iprod{\cdot}{\cdot}$.
  For $T\in \Lin(X)$ on a Hilbert space $X$ we define $\re T = \frac{1}{2}(T+T^\ast)$.
 $H^k(a,b;\R^n)$ is the $k$th order Sobolev space of functions $f:[a,b]\to \R^n$.
For $T\in \Lin(X)$ we denote $T>0$ if $T-\eps I\geq 0$ for some $\eps>0$.
\section{The Main Results for PHS}
\label{sec:PHSSummary}

In this section we summarise our main results for the class~\eqref{eq:PHSintro} of boundary controlled PDEs. 
The parameters $P_2,P_1,P_0,G_0\in\R^{n\times n}$ are assumed to satisfy
$P_2=-P_2^T$, $P_1={P}_1^T$, $P_0=-P_0^T$, $G_0=G_0^T \ge 0$, and
${\mathcal H}(\cdot)$ is a bounded and Lipschitz continuous matrix-valued function such that ${\mathcal H}(z)={\mathcal H}(z)^T$ and ${\mathcal H}(z)\ge\kappa I$,
with $\kappa > 0$, for all $z\in[a,\,b]$. The distributed disturbance input profile is assumed to satisfy $B_d(\cdot)\in \Lp[2](a,b;\R^{n\times n_{d,1}})$ and can be unknown.

We consider first and second order PHS by assuming that either $P_2$ is invertible (the system~\eqref{eq:PHSintro} is of order $N=2$) or $P_2=0$ and $P_1$ is invertible (the system is of order $N=1$).
The boundary inputs and ouputs are determined using the following \emph{boundary port variables}.
\begin{definition} 
  \label{def:BPvariables}
  The \emph{boundary port variables} $f_\partial(t)$ and $e_\partial(t)$ associated to the system \eqref{eq:PHSintro} are defined as
\eq{
  \pmat{f_\partial(t) \\ e_\partial(t)} = R_{ext}\Phi(\mc{H}x(t)), \quad\; \mbox{with} \quad R_{ext}=\frac{1}{\sqrt{2}}\pmat{Q & -Q \\ I & I}
}
where $Q\in \R^{2nN\times 2nN}$ and $\Phi(\cdot):H^{N}(a,b;\R^n)\to \R^{2nN}$ are defined so that
\begin{itemize}
  \item if $N=2$, then 
\eq{
  Q=\pmat{P_1 & P_2 \\ -P_2 & 0}, \qquad 
  \Phi(\mc{H}x):=\pmat{ ({\mathcal H} x)(b)\\\frac{\partial ({\mathcal H} x)}{\partial z}(b)\\ 
    ({\mathcal H} x)(a)\\\frac{\partial ({\mathcal H} x)}{\partial z}(a)}, 
}
whenever  $\mc{H}x\in H^2(a,b;\R^n)$.
  \item if $N=1$, then $Q=P_1$ and $\Phi(\mc{H}x) =  \pmatsmall{\mc{H}x(b)\\ \mc{H}x(a)} $ whenever $\mc{H}x\in H^1(a,b;\R^n)$.
\end{itemize}
\end{definition}

The input $u(t)\in \R^p$, output $y(t)\in \R^p$ (the numbers of inputs and outputs are the same) and the disturbance inputs $\wdist(t)=(\wdistj[1](t),\wdistj[2](t),$ $\wdistj[3](t))^T\in \R^{n_{d,1}+p+n_{d,3}}$ of the system are defined as in~\eqref{eq:PHSintro}. 
We assume the matrices $W_1,W_2$, and $\tilde{W}$ determining the inputs and outputs
satisfy the following (concrete and checkable) conditions.
As shown later in Lemma~\ref{lem:PHSpassiveCond},
part (b) of Assumption~\ref{ass:PHSW} guarantees that~\eqref{eq:PHSintro} is impedance passive.

\begin{assumption}
  \label{ass:PHSW}
  Denote  $\Sigma := \pmatsmall{0&I\\I&0}\in \R^{2nN\times 2nN}$.
We assume
  $W_1\in \R^{p\times 2nN}$ and $W_2\in \R^{n_{d,3}\times 2nN}$ with $n_{d,3}=nN-p$ and $\tilde{W}\in \R^{p\times 2nN}$ satisfy the following
 
\begin{itemize}
\setlength{\itemsep}{1ex}
\item[\textup{(a)}]  $W:=  \pmatsmall{W_1\\ W_2} \in \R^{nN\times 2nN}$ has full rank and $W\Sigma W^T\geq 0$
\item[\textup{(b)}] 
  $\iprod{(W_1^T \tilde{W}+ \tilde{W}^T W_1- \Sigma)g}{g}\geq 0$ for all $g\in\ker(W_2)$.
\end{itemize}
\end{assumption}

Our second assumption concerns stabilizability properties of~\eqref{eq:PHSintro}.
The system~\eqref{eq:PHSintro} is \keyterm{exponentially stable} if there exist $M,\ga>0$ such that
with $u(t)\equiv 0$ and $\wdist(t)\equiv 0$
we have
\eq{
  \norm{x(\cdot,t)}_{\Lp[2](a,b)}\leq M e^{-\ga t} \norm{x(\cdot,0)}_{\Lp[2](a,b)}
}
for all $x(\cdot,0)\in \Lp[2](a,b;\R^n)$ such that $\mc{H}x(\cdot,0)\in H^N(a,b;\R^n)$ and for which~\eqref{eq:PHSintroBCs} hold for $t=0$.

\begin{assumption}
  \label{ass:PHSstab}
  For any $K\in\R^{p\times p}$, $K>0$, system~\eqref{eq:PHSintro} becomes exponentially stable with output feedback $u(t)=-K y(t)$.
\end{assumption}

  The output feedback $u(t)=-Ky(t)$ alters the boundary conditions of the PDE~\eqref{eq:PHSintro} by changing $W_1$ in~\eqref{eq:PHSintroBCs} to $W_1+K \tilde{W}$.
By~\citel{HumPau18}{Lem.~7} 
Assumption~\ref{ass:PHSstab} 
holds in particular
if 
$W_1\in \R^{nN\times 2nN}$ (i.e.,~\eqref{eq:PHSintro} has $p=nN$ inputs) and if Assumption~\ref{ass:PHSW} holds.
For further results on stability of~\eqref{eq:PHSintro}, see~\cite{AugJac14}. 

Definition~\ref{def:Contrparams} contains the construction of the controller~\eqref{eq:controllerintro}. 
The controller 
 has an \keyterm{internal model} 
 of the frequencies in~\eqref{eq:yrefdist}
in the sense that 
$\set{\pm i\gw_k}_{k=1}^q\cup\set{0}$ are 
eigenvalues of $J_c$
 with geometric multiplicities equal to $p$ (see also Section~\ref{sec:RORPbackground}).

\begin{definition}
\label{def:Contrparams}
Given $0<\gw_1<\cdots<\gw_q$ in~\eqref{eq:yrefdist}, choose the parameters of the controller~\eqref{eq:controllerintro}   
 on $X_c=\R^{p(2q+1)}$ so that
$D_c>0$, $\gd_c>0$,
\begin{subequations}%
  \label{eq:Contrparams}
  \eqn{
    J_c &= \blkdiag(J_c^0, J_c^1,\ldots, J_c^q), 
    \\
    J_c^0&=0_p, \quad J_c^k = \pmat{0&\gw_k I_p\\-\gw_k I_p&0} , \\
    B_c & =  \pmat{B_c^0\\\vdots \\ B_c^q}, \qquad B_c^0=I_p, \quad B_c^k=\pmat{I_p\\0}.
  }
\end{subequations}
\end{definition}

The following theorem is the main result of this section.

\begin{theorem}
  \label{thm:PHSmain}
  Let Assumptions~\textup{\ref{ass:PHSW}} and~\textup{\ref{ass:PHSstab}} be satisfied and let $0=\gw_0<\gw_1<\cdots<\gw_q$.
  Assume~\eqref{eq:PHSintro}  has no transmission zeros at $\set{\pm i\gw_k}_{k=0}^q\subset i\R$.
For every $D_c>0$ there exists 
   $\gd_c^\ast>0$ such that for all $\gd_c\in(0,\gd_c^\ast)$
   the controller 
   in Definition~\textup{\ref{def:Contrparams}} 
   achieves output tracking and disturbance rejection for all signals in~\eqref{eq:yrefdist}.
   In particular, there exists $\ga>0$ (depending on $\gd_c\in(0,\gd_c^\ast)$) such that 
  \eqn{
    \label{eq:PHSerrconv}
    e^{\ga t}\norm{y(t)-\yref(t)} \to 0 \qquad \mbox{as} \quad t\to\infty
  }
  for all $\yref(t)$ and $\wdist(t)$ in~\eqref{eq:yrefdist} and for all initial states 
  $x(\cdot,0)\in \Lp[2](a,b;\R^n)$
  and $x_c(0)\in \R^{p(2q+1)}$
  such that $\mc{H}x(\cdot,0)\in H^N(a,b;\R^n)$ and 
  which satisfy
  the boundary conditions~\eqref{eq:PHSintroBCs} at $t=0$.

  The controller is robust in the sense that the tracking~\eqref{eq:PHSerrconv} is achieved (with a modified $\ga>0$) also if the parameters $(P_2,P_1,P_0,G_0,W_1,W_2,\tilde{W},\mc{H},B_d)$ of~\eqref{eq:PHSintro} are perturbed in such a way that Assumption~\textup{\ref{ass:PHSW}} continues to hold and the closed-loop system remains exponentially stable.
  \end{theorem}

The proof of Theorem~\ref{thm:PHSmain} is presented 
in Section~\ref{sec:RORPcontroller}.
If $\gw_0=0$ is a transmission zero, then $J_c^0$ and $B_c^0$ 
can be removed from the controller parameters in~\eqref{eq:Contrparams} and
Theorem~\ref{thm:PHSmain} holds for $\yref(t)$ and $\wdist(t)$ with $a_0=0$ and $b_0=0$.

\section{Background on Boundary Control Systems}
\label{sec:BCS}

Our main abstract results are formulated 
for the general class of 
\keyterm{boundary control and observation systems}~\cite{Sal87a,CheMor03}%
\begin{subequations}
  \label{eq:BCSplant}
  \eqn{
    \dot{x}(t)&=\mf{A}_0 x(t) + B_d \wdistj[1](t), 
    \qquad x(0)=x_0\in Z\\
    \label{eq:BCSplantInput}
    \mf{B}x(t)&=u(t)+\wdistj[2](t)\\
    \mf{B}_dx(t)&=\wdistj[3](t)\\
y(t)&=\mf{C} x (t) 
  }
\end{subequations}
on a Hilbert space $X$.
We present these abstract results \keyterm{only in the case $D_c=0$}. 
  This simplification does not result in loss of generality, because if $D_c\neq 0$, then~\eqref{eq:BCSplantInput} 
  becomes
  \eqn{
    \label{eq:DcRewrite}
    (\mf{B}+D_c\mf{C})x(t)&=\tilde{u}(t)+(\wdistj[2](t) + D_c \yref(t))
  }
  (which has the same structure as~\eqref{eq:BCSplantInput})
  where $\tilde{u}(t)$ is the control produced by the controller~\eqref{eq:controllerintro} with $D_c=0$.
We make the following standard assumptions on the parameters of~\eqref{eq:BCSplant}.

\begin{assumption}
  \label{ass:BCSass}
  We assume  $X$ and $Z\subset X$ are (complex) Hilbert spaces and 
  $\mf{A}_0\in \Lin(Z, X)$, $\mf{B}\in \Lin(Z,\C^p)$, 
    $B_d\in \Lin(\C^{n_{d,1}},X)$,
    $\mf{B}_d\in \Lin(Z,\C^{n_{d,3}})$ and $\mf{C}\in  \Lin(Z, \C^p)$ have the properties:
  \begin{itemize}
      \setlength{\itemsep}{1ex}
    \item[\textup{(a)}] The operator $A:=\mf{A}_0\vert_{\Dom(A)}$ with $\Dom(A)= \ker(\mf{B})\cap \ker(\mf{B}_d)$ generates a contraction semigroup $T(t)$ on $X$.
    \item[\textup{(b)}]  The operator 
	$\pmatsmall{\mf{B}\\ \mf{B}_d}\in
	\Lin(Z,\C^{p+n_{d,3}} )$
      is surjective.
    \item[\textup{(c)}] $\re \iprod{\mf{A}x}{x}\leq \re \iprod{\mf{B}x}{\mf{C}x}_{\C^p}$ for all $x\in Z$.
  \end{itemize}
\end{assumption}

By~\citel{MalSta07}{Thm. 3.4} part (c) of Assumption~\ref{ass:BCSass} is equivalent to the system~\eqref{eq:BCSplant} being impedance passive in the sense that 
\eq{
  \frac{1}{2}\ddb{t}\norm{x(t)}_X^2 \leq \iprod{u(t)}{y(t)}_{\C^p}.
}
We also denote $\mf{A}:=\mf{A}_0\vert_{\Dom(\mf{A})}$ with $\Dom(\mf{A})=  \ker(\mf{B}_d)$, and in this notation we have
  $A=\mf{A}\vert_{\Dom(A)}$ and $\Dom(A)= \Dom(\mf{A})\cap \ker(\mf{B})$.

For $\gl\in\rho( A)$ we denote the transfer function (from the input $u(t)$ to the output $y(t)$) of the system~\eqref{eq:PHSintro} by $P(\gl)$. 
By~\citel{CheMor03}{Thm. 2.9}, for any $u\in U$ and $\gl\in\rho(A)$ we have $P(\gl)u=\mf{C}x$ where $x\in Z$ is such that $(\gl-\mf{A})x=0$ and $\mf{B}x=u$.
If we denote $\re T=\frac{1}{2} (T+T^\ast)$, then the passivity of the system implies that $\re P(i\gw)\geq 0$ for all $i\gw \in \rho( A)\cap i\R$, see~\cite{Sta02}.

We assume the controller~\eqref{eq:controllerintro}
  on $X_c=\C^{n_c}$ satisfies $J_c^\ast = -J_c\in \C^{n_c\times n_c}$, $B_c\in \C^{n_c\times p}$, $D_c\in \C^{p\times p}$ with $D_c\geq 0$ and $\gd_c>0$
    (as mentioned above, in Sections~\ref{sec:BCS}--\ref{sec:RORPcontroller} we let $D_c=0$).
We  now show that the closed-loop system consisting of~\eqref{eq:BCSplant} and the controller~\eqref{eq:controllerintro} on $X_c=\C^{n_c}$ leads to a well-defined closed-loop state $x_e(t):=(x(t),x_c(t))^T$ and regulation error $e(t)$ for all reference and disturbance signals in~\eqref{eq:yrefdist}.
The closed-loop system 
(with $D_c=0$) 
has the form
  \eq{
    \dot{x}_e(t)&= \pmat{\mf{A}_0&0\\-\gd_c B_c \mf{C}&J_c} x_e(t) + \pmat{B_d&0\\0&\gd_c B_c}\pmat{\wdistj[1](t)\\\yref(t)}\\
    \MoveEqLeft\bmat{\mf{B}&-\gd_c B_c^\ast\\\mf{B}_d&0}x_e(t)
    =\pmat{\wdistj[2](t)\\\wdistj[3](t)}\\
    e(t)&=\pmat{\mf{C},\,0}x_e(t) - \yref(t) 
  }
  with state $x_e(t)=(x(t),x_c(t))^T\in X_e:=X\times X_c$.
  We denote
  \eq{
    \mf{A}_e = \pmat{\mf{A}_0&0\\-\gd_c B_c \mf{C}&J_c} , ~
    \mf{B}_e = \pmat{\mf{B} &-\gd_c B_c^\ast\\\mf{B}_d&0}, 
  }
  $B_e=\pmatsmall{B_d& 0\\0& \gd_c B_c}
$,
  and $\mf{C}_e=\pmat{\mf{C},\, 0}$.

\begin{proposition}
\label{prop:CLexistence}
Under Assumption~\textup{\ref{ass:BCSass}} and
for $J_c^\ast=-J_c$ and $D_c=0$ the operator $A_e:=\mf{A}_e\vert_{\ker(\mf{B}_e)}$ generates a strongly continuous contraction semigroup $T_e(t)$ on $X_e$.
For any $\yref(\cdot)\in C^2([0,\infty);\C^p)$ and $\wdist(\cdot)\in C^2([0,\infty);\C^{n_{d,1}+p+n_{d,3}})$
and for all initial states $x(0)\in Z$
and $x_c(0)\in X_c$ satisfying the compatibility conditions $\mf{B} x(0)=\gd_c B_c^\ast x_c(0)+\wdistj[2](0) $ and $\mf{B}_dx(0)=\wdistj[3](0)$ the closed-loop system has a state
  \eq{
    x(\cdot)&\in C(0,T;Z)\cap C^1(0,T;X), \qquad
    x_c(\cdot)\in C^1(0,T; X_c)
  }
  and $e(t)=y(t)-\yref(t)\in C(0,T;\C^p)$ for all $T>0$.  
\end{proposition}

\begin{proof}
  The closed-loop system is a boundary control and observation system on the spaces $Z\times X_c$ and $X_e=X\times X_c$. The operator $\mf{B}_e$ is surjective due to Assumption~\ref{ass:BCSass}(b).
  Our aim is to show that $A_e$ generates a contraction semigroup on $X_e$.
  Since $\mf{C}_e\in \Lin(Z\times X_c,\C^p)$ and $B_d$ and $B_c$ are bounded, 
  the properties of the closed-loop system's state then follow (due to linearity) from~\citel{TucWei09book}{Prop. 4.2.10 and Prop. 10.1.8}.
We now use the Lumer--Phillips Theorem.
Let $x_e:=(x,x_c)^T\in \ker(\mf{B}_e)$.
Then $\mf{B}x=\gd_c B_c^\ast x_c$ and $\mf{B}_dx=0$. In particular $x\in \Dom(\mf{A})$ and $\mf{A}_0x=\mf{A}x$.  
  The impedance passivity of $(\mf{A},\mf{B},\mf{C})$ implies $\re \iprod{\mf{A}x}{x}\leq \re \iprod{\mf{B}x}{\mf{C}x}$ for all $x\in Z$~\citel{MalSta07}{Thm. 3.4}.
  Thus 
  \eq{
    \re \iprod{A_e x_e }{x_e}
    &= \re \iprod{\mf{A}x }{x} + \re \iprod{J_c x_c - \gd_c B_c \mf{C}x}{x_c}\\
    & \leq \re \iprod{\mf{B}x }{\mf{C}x} - \re \iprod{ \mf{C}x}{\gd_c B_c^\ast x_c}=0,
  }
  since $J_c$ is skew-adjoint and $\gd_c B_c^\ast x_c = \mf{B}x$. Therefore $A_e$ is dissipative, and it remains to show that $\gl-A_e$ is surjective for some $\gl>0$. 
  Let $\gl>0$, $y_1\in X$, and $y_2\in X_c$ be arbitrary. We will construct $x_e:=(x,x_c)^T\in \ker(\mf{B}_e)$ such that $(y_1,y_2)^T=(\gl-A_e)x_e$.
  Recall that $P(\gl)$ is the transfer function of $(\mf{A},\mf{B},\mf{C})$ 
  and denote $P_c(\gl)=\gd_c^2 B_c^\ast R(\gl,J_c) B_c$.
Since $\gl>0$ is real, we have $P_c(\gl)\geq 0$ and $P(\gl)\geq 0$, and it can be shown that $Q_1(\gl):=I+P(\gl)P_c(\gl)$ and $Q_2(\gl):=I+P_c(\gl)P(\gl)$ are boundedly invertible.  
  Denote $R_\gl=R(\gl,A)$ and $R_\gl^c = R(\gl,J_c)$ for brevity.
  Due to the theory in~\cite{CheMor03},~\citel{TucWei09book}{Ch. 10} the  ``abstract elliptic problem'' 
  \eq{
    (\gl-\mf{A})x& = y_1 \\
    \mf{B}x&= 
    Q_2(\gl)\inv
    (\gd_c B_c^\ast R_\gl^c y_2-P_c(\gl)\mf{C}R_\gl y_1)
  }
  has a solution $x\in Z$. Now~\citel{CheMor03}{Thm. 2.9} and linearity imply
  \eq{
    \mf{C} x &= \mf{C}R_\gl y_1  
    + P(\gl)Q_2(\gl)\inv
    (\gd_c B_c^\ast R_\gl^c y_2-P_c(\gl)\mf{C}R_\gl y_1)\\
    &=  
    Q_2(\gl)\inv 
    (\mf{C}R_\gl y_1+\gd_c P(\gl)B_c^\ast R_\gl^c y_2).
  }
  If we now define
  \eq{
    x_c= R_\gl^c y_2  -\gd_c  R_\gl^c B_c
    Q_1(\gl)\inv
    (\mf{C}R_\gl y_1 + \gd_c P(\gl)B_c^\ast R_\gl^c y_2),
  }
  then 
  \eq{
    \gd_c B_c^\ast x_c &= \gd_c B_c^\ast R_\gl^c y_2  \\
 &\quad   - P_c(\gl)Q_1(\gl)\inv
    (\mf{C}R_\gl y_1 + \gd_c P(\gl)B_c^\ast R_\gl^c y_2)\\
    &= 
    Q_2(\gl)\inv
    ( \gd_c B_c^\ast R_\gl^c y_2-P_c(\gl)\mf{C}R_\gl y_1 ) = \mf{B}x
  }
  and thus $x_e:=(x,x_c)^T$ satisfies $\mf{B}_ex_e=0$. A direct computation also shows that 
  \ieq{
    -\gd_c B_c \mf{C}x + (\gl-J_c)x_c = y_2,
  }
  and thus indeed $(y_1,y_2)^T = (\gl-A_e)x_e$.
\end{proof}

\section{Robust tracking and disturbance rejection}
\label{sec:RORPbackground}

  In this section we formulate the robust output regulation problem and present a general condition for a controller~\eqref{eq:controllerintro} to solve this problem.

\begin{RORP}
Let $0<\gw_1<\cdots<\gw_q$.
  Choose a controller~\eqref{eq:controllerintro} in such a way that the following hold.
\begin{itemize}
  \setlength{\itemsep}{.5ex}
  \item[\textup{(a)}] 
    The semigroup $T_e(t)$ generated by $A_e=\mf{A}_e\vert_{\ker(\mf{B}_e)}$ is exponentially stable.
  \item[\textup{(b)}] 
    There exists $\ga>0$ such that for all $\yref(t)$ and $\wdist(t)$ of the form~\eqref{eq:yrefdist} and for all
    initial states $x(0)\in Z$ and $x_c(0)\in X_c$
    satisfying  the boundary conditions of~\eqref{eq:BCSplant}
  the regulation error satisfies
\eq{
e^{\ga t}\norm{y(t)-\yref(t)} \to 0 \qquad \mbox{as} \quad t\to\infty.
}
\item[\textup{(c)}] If
$(\mf{A}_0,\mf{B},\mf{B}_d,B_d,\mf{C})$
in~\eqref{eq:BCSplant} 
are perturbed in such a way that Assumption~\textup{\ref{ass:BCSass}} is satisfied and
the perturbed closed-loop operator generates an exponentially stable semigroup, 
  then \textup{(b)} continues to hold for some $\tilde{\ga}>0$.
\end{itemize}
\end{RORP}

The robust output regulation problem only has a solution if
 the control system does not have transmission zeros at  
 $\set{\pm i\gw_k}_{k=0}^q$ (a transmission zero at $\gl\in\rho(A)$ is equivalent to $P(\gl)\in \C^{p\times p}$ being singular).
 For impedance passive systems it is natural to make the following stronger assumption. 

\begin{assumption}
  \label{ass:Piwksur}
Let $0=\gw_0<\gw_1<\cdots<\gw_q$.
We assume 
    $\pm i\gw_k\in \rho( A)$ and $\re P(\pm i\gw_k)>0$ for all $k\in \List[0]{q}$.
\end{assumption}

The following theorem shows that a controller incorporating an \emph{internal model} 
(in the sense of conditions~\eqref{eq:Gconds} below)
will solve the robust output regulation problem provided that the closed-loop system is exponentially stable. The result generalises~\citel{HumPau18}{Thm.~4} 
  by removing the assumption of regularity (and well-posedness) of the closed-loop system,
  and the proof is completed without reformulating~\eqref{eq:BCSplant} as a system 
  with extended state and
  distributed inputs.

\begin{theorem}
  \label{thm:IMP}
Let $0=\gw_0<\gw_1<\cdots<\gw_q$.
  A controller~\eqref{eq:controllerintro} with $J_c^\ast=-J_c$, $D_c=0$ and $\gd_c>0$ solves the robust output regulation problem if 
  $A_e = \mf{A}_e\vert_{\ker(\mf{B}_e)}$ generates an exponentially stable semigroup and 
  \begin{subequations}
    \label{eq:Gconds}
    \eqn{
      \label{eq:Gconds1}
      \ran(\pm i\gw_k - J_c)\cap \ran(B_c) &= \set{0}, \quad \forall k\in \List[0]{q}\\
      \label{eq:Gconds2}
      \ker(B_c)&= \set{0}.
    }
  \end{subequations}
Then there exists $\ga>0$ such that
\eq{
  e^{\ga t}\norm{y(t)-\yref(t)}\to 0, \qquad \mbox{as} \quad t\to\infty
}
for any $\yref(t)$ and $\wdist(t)$ of the form~\eqref{eq:yrefdist} and for all 
$x(0)\in Z$
and $x_c(0)\in X_c$ satisfying the compatibility conditions $\mf{B}x(0)=\gd_c B_c^\ast x_c(0)+\wdistj[2](0)$ and $\mf{B}_dx(0)=\wdistj[3](0)$.
\end{theorem}

\begin{proof}
  Assume the closed-loop system is exponentially stable and~\eqref{eq:Gconds} are satisfied.  
Then there exist $M_e,\gw_e>0$ such that $\norm{T_e(t)}\leq M_ee^{-\gw_e t}$.
  Let $\set{\mu_k}_{k=-q}^q$ be such that 
  $\mu_k = \gw_{k}$ for $k>0$, $\mu_0=0$, and 
  $\mu_k = -\gw_{\abs{k}}$ for $k<0$. We can then write
  \eq{
    \yref(t)=\sum_{k=-q}^q \yrk e^{i\mu_k t}, 
    \qquad \wdist(t) = \sum_{k=-q}^q \pmat{\wdkone\\\wdktwo\\\wdkthree} e^{i\mu_k t}
  }
  for some constant elements $\set{\yrk}_k$, $\set{\wdkone}_k$, $\set{\wdktwo}_k$, and $\set{\wdkthree}_k$.
  Since $i\mu_k\in\rho(A_e)$ for all $k$, we have from~\citel{TucWei09book}{Sec. 10.1} that we can choose $\Sigma_k\in Z$ such that
  \begin{subequations}
    \label{eq:Sigmak}
    \eqn{
    \label{eq:Sigmak1}
      (i\mu_k-\mf{A}_e)\Sigma_k &= B_e \pmat{\wdkone\\\yrk}\\
    \label{eq:Sigmak2}
      \mf{B}_e \Sigma_k &= \pmat{\wdktwo\\\wdkthree}.
    }
  \end{subequations}
Consider initial conditions
$x(0)\in Z$ 
and $x_c(0)\in X_c$ satisfying the compatibility conditions $\mf{B}x(0)=\gd_c B_c^\ast x_c(0)+\wdistj[2](0)$ and $\mf{B}_dx(0)=\wdistj[3](0)$.
If we define $\Sigma(t) = \sum_{k=-q}^q e^{i\mu_k t}\Sigma_k\in Z$, then
\eq{
  \ddb{t}(x_e(t)-\Sigma(t))
  &= \mf{A}_e x_e(t)+ B_e \pmat{\wdistj[1](t)\\\yref(t)} - \sum_{k=-q}^q i\mu_k e^{i\mu_k t}\Sigma_k\\
  &= \mf{A}_e (x_e(t) - \Sigma (t))
}
due to~\eqref{eq:Sigmak1}. For all $t\geq 0$ we also have from~\eqref{eq:Sigmak2} that
\eq{
  \mf{B}_e(x_e(t)-\Sigma(t)) 
  = \pmat{\wdistj[2](t)\\\wdistj[3](t)} - \sum_{k=-q}^q e^{i\mu_k t} \mf{B}_e\Sigma_k=0.
}
 Thus $x_e(t)-\Sigma(t)\in \Dom(A_e)$ is a classical solution of the abstract Cauchy problem $\ddb{t} (x_e(t)-\Sigma(t))=A_e(x_e(t)-\Sigma(t)) $, and therefore  $\norm{x_e(t)-\Sigma(t)}=\norm{T_e(t)(x_e(0)-\Sigma(0))}\leq M_ee^{-\gw_e t} \norm{x_e(0)-\Sigma(0)}$.

If we write $\Sigma_k = 
\pmatsmall{\Pi_k\\\Gamma_k}
\in Z\times X_c$, then~\eqref{eq:Sigmak1} and the conditions~\eqref{eq:Gconds} imply 
\eq{
  & \pmat{i\mu_k-\mf{A}_0&0\\\gd_c B_c \mf{C}&i\mu_k-J_c}\pmat{\Pi_k\\\Gamma_k}= \pmat{B_d\wdkone\\\gd_c B_c \yrk} \\[1ex]
 \Rightarrow \quad &(i\mu_k-J_c)\Gamma_k = \gd_c B_c(\yrk-\mf{C}\Pi_k)\\
 \stackrel{\eqref{eq:Gconds1}}{\Rightarrow} \quad & B_c(\yrk-\mf{C}\Pi_k)=0 \quad
 \stackrel{\eqref{eq:Gconds2}}{\Rightarrow} \quad  \yrk=\mf{C}\Pi_k = \mf{C}_e\Sigma_k.
}
Using $\mf{C}_e\Sigma_k=\yrk$, we can write  $e(t)=y(t)-\yref(t)$
as
\eq{
  e(t)
  &= \mf{C}_ex_e(t) - \sum_{k=-q}^q \yrk e^{i\mu_kt} 
   = \mf{C}_ex_e(t) - \sum_{k=-q}^q \mf{C}_e\Sigma_k e^{i\mu_kt} \\
  &= \mf{C}_e(x_e(t) - \Sigma(t)).
}
Finally, since $\mf{C}_e A_e\inv\in \Lin(X,\C^p)$ for boundary control systems, we have
\eq{
  \norm{e(t)} 
  &=\norm{\mf{C}_e(x_e(t) - \Sigma(t))}
  \\ &
 =\norm{\mf{C}_eA_e\inv T_e(t)A_e(x_e(0) - \Sigma(0))}\\
  &\leq M_e e^{-\gw_e t}\norm{\mf{C}_eA_e\inv} \cdot \norm{A_e(x_e(0) - \Sigma(0))} 
}
and thus~$e^{\ga t}\norm{e(t)}\to 0$ as $t\to\infty$ for any $0<\ga <\gw_e$.

  Since the proof can be repeated analogously for any 
  perturbations of $(\mf{A}_0,\mf{B},\mf{B}_d,B_d,\mf{C})$ for which Assumption~\ref{ass:BCSass} is satisfied and the closed-loop semigroup is exponentially stable,
  the controller satisfies
   part (c) of the robust output regulation problem.
\end{proof}

\section{A Passive Robust Controller}
\label{sec:RORPcontroller}

In this section we prove that if the system~\eqref{eq:BCSplant} is exponentially stable and
the parameters of the controller~\eqref{eq:controllerintro} 
on $X_c =\C^{p(2q+1)}$
are chosen as (real) matrices $D_c=0$,
\begin{subequations}%
\label{eq:ContrparamsProofs}
\eqn{
J_c &= \blkdiag(J_c^0, J_c^1,\ldots, J_c^q), 
\\
J_c^0&=0_p, \quad J_c^k = \pmat{0&\gw_k I_p\\-\gw_k I_p&0} , \\
B_c & =  \pmat{B_c^0\\\vdots \\ B_c^q}, \qquad B_c^0=I_p, \quad B_c^k=\pmat{I_p\\0},
}
\end{subequations}
then the controller solves the robust output regulation problem for a range of gain parameters $\gd_c>0$.
The following theorem is the main abstract result of the paper,
  and it is also used in proving Theorem~\ref{thm:PHSmain} at the end of this section.

\begin{theorem}
  \label{thm:CLstabpert}
    Let $0=\gw_0<\gw_1<\cdots<\gw_q$. 
  Assume $A$ generates an exponentially stable semigroup $T(t)$, $C:= \mf{C}\vert_{\Dom(A)}$ is admissible with respect to $T(t)$, and Assumption~\textup{\ref{ass:Piwksur}} holds.
  Then there exists $\gd_c^\ast>0$ such that
  for all 
$\gd_c\in(0,\gd_c^\ast)$
the controller~\eqref{eq:controllerintro} 
on $X_c=\C^{p(2q+1)}$
with parameters~\eqref{eq:ContrparamsProofs} and $D_c=0$ solves the robust output regulation problem for all
$\yref(t)$ and $\wdist(t)$ in~\eqref{eq:yrefdist}.
\end{theorem}

The main part of the proof of Theorem~\ref{thm:CLstabpert} 
consists of showing the exponential stability of the closed-loop system
for $\gd_c\in(0,\gd_c^\ast)$, and for this we use a new Lyapunov argument.
Similar methods have been used in study of stability of coupled PHS especially in~\cite{RamLeG14,MacCal18}. Our situation is different from the previous references due to the fact that the infinite-dimensional system~\eqref{eq:BCSplant} is exponentially stable and the unstable controller~\eqref{eq:controllerintro}
  is finite-dimensional.  
The proof of
Theorem~\ref{thm:CLstabpert} begins with the definition of a component
$H\in \Lin(X_c,X)$ of the Lyapunov candidate function 
in 
Lemma~\ref{lem:IMstabpert}.
For the proofs we 
define a block-diagonal similarity transform $T=\blkdiag(T_0,T_1,\ldots,T_q)\in\C^{n_c\times n_c}$ where $n_c=p(2q+1)$ such that for $k\in \List{q}$ 
\eq{
  T_0=I_p, \qquad
  T_k = \pmat{I&I\\iI&-iI}, \quad T_k\inv = \frac{1}{2} \pmat{I&-iI\\I&iI}.
}
Moreover, we define $G_1 = T\inv J_c T\in \C^{p(2q+1)\times p(2q+1)}$ and $G_2 = T\inv B_c\in \C^{p\times p(2q+1)}$.
A direct computation shows that 
\eq{
  G_1  &= \blkdiag(i\gw_0 I_p,i\gw_1I_p,-i\gw_1 I_p,\ldots, i\gw_q I_p,-i\gw_q I_p)\\
  G_2 &= \frac{1}{2}\pmat{I_p,I_p,\ldots,I_p}^T .
}

\begin{lemma}
  \label{lem:IMstabpert}
  Let Assumption~\textup{\ref{ass:Piwksur}} hold and
  assume $A$ generates an exponentially stable semigroup on $X$.
Let $X_c=\C^{p(2q+1)}$ and let
$J_c$ and $B_c$ be as in~\eqref{eq:ContrparamsProofs}.
Then there exists 
$H\in \Lin(X_c,X)$ satisfying
$\ran(H)\subset Z$ such that  
\eqn{
  \label{eq:SylEqn}
HJ_c = \mf{A} H\qquad \mbox{and} \qquad \mf{B}H=-B_c^\ast ,
}
and we have $\mf{C}H\in \Lin(X_c,\C^p)$.
  Moreover, there exist constants $\gd_0^\ast ,M_c>0$ such that for any $\gd_c\in (0,\gd_0^\ast)$ we can choose $\P_{c0}>0$ such that $\norm{P_{c0}}\leq M_c$ and
  \eq{
    P_{c0}(J_c+ \gd_c^2 B_c \mf{C} H) + (J_c+ \gd_c^2 B_c \mf{C} H)^\ast P_{c0} = - \gd_c^2I.
  } 
\end{lemma}

\begin{proof}
Since $J_c = TG_1T\inv$, an operator $H\in \Lin(X_c,X)$ with $\ran(H)\subset Z$ satisfies~\eqref{eq:SylEqn}
if and only if
$HTG_1 = \mf{A} HT$ and $\mf{B}HT=-B_c^\ast T $. Due to the block-diagonal structure of $G_1$, 
the operator $HT$ has the form 
$HT=(H_0,H_1,H_{-1},\ldots,H_{q},H_{-q})$.
Since $B_c^\ast T =  \pmat{I,\ldots,I}$, 
for each $k\in \List[0]{q}$ the operators $H_{\pm k}:\C^p\to X$ are determined by $z_{\pm k}=H_{\pm k} y$ for all $y\in \C^p$ where $z_{\pm k}$ are the solutions of the abstact elliptic equations
  \eq{
    \begin{cases}
      (\pm i\gw_k-\mf{A})z_{\pm k}=0\\
      \mf{B}z_{\pm k}=- y.
    \end{cases}
  }
  By \citel{TucWei09book}{Prop. 10.1.2, Rem. 10.1.3 \& 10.1.5} the above equations have unique solutions and $H_{k}\in \Lin(\C^p,X)$ and $\ran(H_k)\subset Z$ for all $k\in \List[-q]{q}$. 
  Thus $H\in \Lin(X_c,X)$ and $\ran(H)\subset Z$. 
  We further have from~\citel{CheMor03}{Thm. 2.9} that  $\mf{C}H_{\pm k}y = \mf{C}z_{\pm k}=- P(\pm i\gw_k)y$ for all $y\in \C^p$ and $k\in \List[0]{q}$. Because of this, we have
\eq{
\mf{C} HT
&=-  \Bigl[P(i\gw_0),P(i\gw_1),P(-i\gw_1),\ldots,
P(i\gw_q),P(-i\gw_q)\Bigr],
}
which in particular implies $\mf{C}H\in \Lin(X_c,\C^p)$.

To prove the second claim, 
we first note that Assumption~\ref{ass:Piwksur} implies that $\re\gl>0$ for all 
$\gl\in\gs(P(\pm i\gw_k))$ and $k$. Indeed, if $k\in \List[0]{q}$ and  $\re\gl\leq 0$, then $\re P(\pm i\gw_k)>0$ implies 
\ieq{
\re (P(\pm i\gw_k)-\gl) =  \abs{\re \gl}+\re P(\pm i\gw_k)>0,
}
and thus $P(\pm i\gw_k)-\gl$ is nonsingular.

In the next step we use the results in~\citel{HamPoh11}{App.~B} to show that there exist constants
$M_0,\gw_0,\gd_0^\ast>0$ 
such that
  \eqn{
    \label{eq:IMstabbound}
    \norm{\exp( (J_c+\gd_c^2 B_c \mf{C} H)t)}\leq M_0 e^{-\gw_0 \gd_c^2 t}
  }
  for all
 $\gd_c\in (0,\gd_0^\ast)$ and
  $t\geq 0$.
  If we denote $K=-\mf{C} HT$,
  then
  \eq{
    J_c + \gd_c^2 B_c \mf{C}H
    = T(G_1-\gd_c^2 G_2 K)T\inv.
  }
  Now $K=[K_0,K_1,K_2,\ldots,K_{2q}]$ where $\re\gl<0$ for all $\gl\in\gs(K_k)$ and $k\in \List[0]{2q}$, and $G_2=\frac{1}{2}[I,\ldots,I]^T$. Thus $(G_1-\gd_c^2 G_2 K)^\ast = G_1^\ast - \gd_c^2 K^\ast G_2^\ast$ is 
of the form of $A_c(\eps)$ in~\cite[App. B]{HamPoh11} with $\eps=\gd_c^2/2$.
The proof of Theorem~1 in~\cite[App. B]{HamPoh11} shows that 
there exist $M_1,\gw_0,\gd_0^\ast>0$
such that 
$\norm{\exp((G_1^\ast - \gd_c^2 K^\ast G_2^\ast)t)}\leq M_1 e^{-\gw_0 \gd_c^2 t}$ for all
$\gd_c\in (0,\gd_0^\ast)$ and
$t\geq 0$.
This further implies that if we define $M_0=M_1\norm{T}\norm{T\inv}$, then~\eqref{eq:IMstabbound} holds for all $\gd_c\in (0,\gd_0^\ast)$ and $t\geq 0$.

Let $\gd_c\in (0,\gd_0^\ast)$ and denote $T_{\gd_c}(t)=\exp( (J_c+\gd_c^2 B_c \mf{C} H)t)$ for brevity.
  Since $J_c+\gd_c^2 B_c \mf{C} H$ is Hurwitz, we can choose $\tilde{P}_{c0}>0$ such that
  \eq{
    (J_c+\gd_c^2 B_c \mf{C} H)\tilde{P}_{c0}  +(J_c+\gd_c^2 B_c \mf{C} H)^\ast \tilde{P}_{c0} = -I.
  }
 Here $\tilde{P}_{c0}=\int_0^\infty T_{\gd_c}(t)^\ast T_{\gd_c}(t)dt$, and thus~\eqref{eq:IMstabbound} implies
  \eq{
    \norm{\tilde{P}_{c0}}\leq \int_0^\infty \norm{T_{\gd_c}(t)}^2 dt
    \leq M_0^2 \int_0^\infty e^{-2\gw_0 \gd_c^2 t}dt
    =  \frac{M_0^2}{2\gw_0 \gd_c^2 }.
  }
  Now the matrix $P_{c0} := \gd_c^2 \tilde{P}_{c0}$ has the required properties.
\end{proof}

\textit{Proof of Theorem~\textup{\ref{thm:CLstabpert}}.}
  The proof of~\citel{HamPoh10}{Lem.~12} shows that $\ran(\pm i\gw_k-G_1)\cap \ran(G_2)=\set{0}$ for all $k\in \List[0]{q}$ and $\ker(G_2)=\set{0}$, and by similarity the pair $(J_c,B_c) = (TG_1T\inv,TG_2)$ satisfies the conditions~\eqref{eq:Gconds}. 
By Theorem~\ref{thm:IMP}
  it is thus sufficient to show that the closed-loop system is exponentially stable 
(in the case $\yref(t)\equiv 0$ and $\wdist(t)\equiv 0$).

Let $H\in \Lin(X_c,X)$ and $\gd_0^\ast,M_c>0$
be as in  Lemma~\ref{lem:IMstabpert}, and let $\gd_c\in(0,\gd_0^\ast)$.
We choose
the Lyapunov function candidate $V_e$ for the closed-loop system by 
\eq{
V_e 
&= \iprod{x+\gd_c Hx_c}{\P (x+\gd_c Hx_c)}_X + \iprod{x_c}{\P_c x_c}_{X_c}
}
where $x=x(t)$  and $x_c=x_c(t)$ are the states of the plant and the controller, respectively, and $P$ and $P_c$ will be chosen later. Since the coordinate transform $(x,x_c)\to (x+\gd_c Hx_c,x_c)$ is boundedly invertible, $V_e$ is a valid Lyapunov function candidate whenever $P>0$ and $P_c>0$.

Let $(x(t),x_c(t))^T$ be a classical solution of the closed-loop system with $\yref(t)\equiv 0$ and $\wdist(t)\equiv 0$. 
Since $\mf{B}x(t)=\gd_c B_c^\ast x_c(t)$ and $\mf{B}H=-B_c^\ast$, we have $\mf{B}(x(t)+\gd_c Hx_c(t))=0$. Thus $x(t)+\gd_c Hx_c(t)\in \ker(\mf{B})=\Dom(A) $ and $\mf{A}(x(t)+\gd_c Hx_c(t))=A(x(t)+\gd_c Hx_c(t))$.
If we denote $\tilde{A}=A -\gd_c^2 HB_c C: \Dom(A)\subset X\to X$, then a direct computation using~\eqref{eq:SylEqn} shows that
  \eq{
    \MoveEqLeft[1]\frac{1}{2}\dot{V}_e 
    = \re\iprod{\dot{x}+\gd_c H\dot{x}_c}{\P (x+\gd_c Hx_c)} + \re\iprod{\dot{x}_c}{\P_c x_c} \\
    &= \re\iprod{\mf{A}x+\gd_c HJ_c x_c - \gd_c^2 HB_c \mf{C} x }{\P (x+\gd_c Hx_c)}\\
    &\quad  + \re\iprod{J_c x_c - \gd_c B_c \mf{C} x }{\P_c x_c} \\
   &= \re\iprod{\tilde{A}(x+\gd_c Hx_c)  }{\P (x+\gd_c Hx_c)}
   + \re\iprod{(J_c  + \gd_c^2 B_c \mf{C} H) x_c}{\P_c x_c} \\
   &\quad +\re\iprod{\gd_c^2 B_c \mf{C} Hx_c}{\gd_c H^\ast \P (x+\gd_c Hx_c)}
   \\
   & \quad - \re\iprod{  C (x+\gd_c Hx_c) }{\gd_c B_c^\ast\P_c x_c} .
  }

  Since $A$ generates an exponentially stable semigroup $T(t)$ on $X$, there exists a unique $P_1\in \Lin(X)$ with $P_1>0$ such that $A^\ast P_1 + P_1A = -2I$.
  Moreover, the exponential stability also implies that $C$ is infinite-time admissible with respect to $T(t)$, and by~\citel{TucWei09book}{Thm. 5.1.1} there exists $P_2\in \Lin(X)$ with $P_2\geq 0$ such that $2\re \iprod{Ax_1}{P_2x_1}=-2 \norm{Cx_1}^2$ for all $x_1\in \Dom(A)$.
  Thus if we define $P=P_1+P_2\in \Lin(X)$, then $P>0$ and 
\eq{
2\re \iprod{A x_1}{Px_1}= -2 \norm{x_1}^2-2 \norm{C x_1}^2 \quad\; \forall x_1\in \Dom(A).
}
  The scalar inequality $2ab\leq a^2+b^2$ implies that if $x_1\in \Dom(A)$, then
\eq{
   2\re \iprod{\tilde{A} x_1}{Px_1}
&=2\re \iprod{A x_1}{Px_1} - 2\gd_c^2 \re \iprod{C x_1}{B_c^\ast H^\ast Px_1}\\
&\leq -2 \norm{x_1}^2-2 \norm{C x_1}^2 + \gd_c^2 \norm{C x_1}^2 + \gd_c^2 \norm{PHB_c}^2 \norm{x_1}^2\\
&= -(2-\gd_c^2 \norm{PHB_c}^2) \norm{x_1}^2-(2-\gd_c^2) \norm{C x_1}^2 \\
&\leq - \norm{x_1}^2- \norm{C x_1}^2 
}
whenever
$0<\gd_c\leq \gd_1^\ast$ with  $\gd_1^\ast:=\min \set{1,1/\norm{PHB_c}}>0$.

  Since $\gd_c\in(0,\gd_0^\ast)$ by assumption, 
  we can choose $P_{c0}>0$ (corresponding to this $\gd_c$) 
  as in 
  Lemma~\ref{lem:IMstabpert} and
  define $\P_c = \eps_c \P_{c0}>0$ for some $\eps_c>0$. 
  Then
  $\norm{\P_c}\leq M_c \eps_c$ and 
  \eq{
    \P_c(J_c + \gd_c^2 B_c \mf{C} H) + (J_c + \gd_c^2 B_c \mf{C} H)^\ast \P_c = -\eps_c\gd_c^2I.
  }
If $0<\gd_c<\min \set{\gd_0^\ast,\gd_1^\ast}$, we can estimate (using the inequality 
$2\re \iprod{z_1}{z_2}\leq 2 \norm{z_1}\norm{z_2}\leq \frac{1}{2} \norm{z_1}^2 + 2 \norm{z_2}^2$ 
in the last term)
\eq{
\dot{V}_e
 &  = 2\re\iprod{\tilde{A}(x+\gd_c Hx_c)  }{\P (x+\gd_c Hx_c)}
     + 2\re\iprod{(J_c  +\gd_c^2 B_c \mf{C} H) x_c}{\P_c x_c} \\
   &\quad +2 \re\iprod{\gd_c^2 B_c \mf{C} Hx_c}{\gd_c H^\ast \P (x+\gd_c Hx_c)}\\
   &\quad  - 2\re\iprod{  C (x+\gd_c Hx_c) }{\gd_c B_c^\ast\P_c x_c} 
   \\
&\leq
- \norm{x+\gd_c Hx_c}^2- \norm{C (x+\gd_c Hx_c)}^2  
      - \eps_c\gd_c^2  \norm{x_c}^2
    + \gd_c^4\norm{B_c \mf{C} H x_c}^2 \\
&\quad 
     + \gd_c^2 \norm{H^\ast \P(x+\gd_c Hx_c)}^2
 + \frac{1}{2} \norm{C (x+\gd_c Hx_c)}^2 + 2 \gd_c^2 \norm{B_c^\ast \P_c x_c}^2\\
    &= \Bigl[ -1 + \gd_c^2 \norm{\P H}^2  \Bigr] \norm{x+\gd_c Hx_c}^2 -\frac{1}{2} \norm{C(x+\gd_c Hx_c)}^2
    \\
    &\quad + \gd_c^2\Bigl[ -\eps_c + \gd_c^2 \norm{B_c \mf{C} H}^2 + 2 M_c^2 \eps_c^2 
      \norm{B_c}^2
    \Bigr]\norm{x_c}^2 .
}
   We can now choose a sufficiently small fixed $\eps_c>0$ and $\gd_2^\ast >0$ such that  if  $0<\gd_c< \gd_c^\ast :=\min \set{\gd_0^\ast,\gd_1^\ast,\gd_2^\ast}$, then
  \eq{
  \dot{V}_e
   \leq -\tilde{\eps}_e \left( \norm{x+\gd_c Hx_c}^2+ \norm{x_c}^2 \right)
   \leq -\tilde{\eps}_e \max \set{\norm{P\inv},\norm{P_c\inv}}V_e
  = : -\eps_e V_e,
  }
  where
  $\eps_e>0$ depends on the choice of $\gd_c>0$.
    Since $T_e(t)$ is contractive, this proves exponential closed-loop stability.
\hfill$\square$

We now present the proof of Theorem~\ref{thm:PHSmain} for PHS.
To use Theorem~\ref{thm:CLstabpert} we formulate~\eqref{eq:PHSintro} as a boundary control system on $X=\Lp[2](a,b ;\C^n)$  with norm defined by $\norm{x}_{\mc{H}}=\sqrt{\iprod{\mc{H}x}{x}_{\Lp[2]}}$ for $x\in X$ (since
$(P_2,P_1,P_0,G_0,\mc{H},W,\tilde{W})$ are real,
real-valued initial data for~\eqref{eq:PHSintro} and~\eqref{eq:controllerintro} leads to real-valued solutions).
We begin by showing that the condition (b) in Assumption~\ref{ass:PHSW} implies impedance passivity of~\eqref{eq:PHSintro}.

\begin{lemma}
  \label{lem:PHSpassiveCond}
  If Assumption~\textup{\ref{ass:PHSW}} holds
and $\wdist(t)\equiv 0$, then the classical solutions of~\eqref{eq:PHSintro} satisfy
\ieq{
\frac{1}{2}\ddb{t}\norm{x(t)}_{\mc{H}}^2 \leq u(t)^T y(t).
}
\end{lemma}

\begin{proof}
Let $\wdist(t)\equiv 0$.
The proof of 
\citel{LeGZwa05}{Thm.~4.2} and (b) imply that the solution of~\eqref{eq:PHSintro} satisfies
\eq{
\MoveEqLeft[1]\frac{1}{2}\ddb{t}\int_a^b x(z,t)^T \mathcal{H}(z) x(z,t)dz
  \leq
  \frac{1}{2} \pmat{f_\partial(t)\\e_\partial (t)}^T \Sigma \pmat{f_\partial(t)\\e_\partial (t)}\\
  &\leq\frac{1}{2} \pmat{f_\partial(t)\\e_\partial (t)}^T (W_1^T \tilde{W}+ \tilde{W}^T W_1) \pmat{f_\partial(t)\\e_\partial (t)}
  =
  y(t)^T u(t)
}
where we have used  that $ \pmatsmall{f_\partial(t)\\ e_\partial(t)} \in\ker( W_2)$ by~\eqref{eq:PHSintroBCs}.
\end{proof}

As shown in~\citel{LeGZwa05}{Sec. 4--5}, \eqref{eq:PHSintro} becomes a boundary control system~\eqref{eq:BCSplant} on $X$  
with choices
  \eq{
    \mf{A}_0 x&:=
    P_2\frac{\partial^2 }{\partial z^2}\big({\mathcal H}x\big)
    +P_1\frac{\partial }{\partial z}\big({\mathcal H}x\big)
    +(P_0-G_0)\left({\mathcal H}x\right)\\
    \Dom(\mf{A}_0)&= Z:= \setm{x\in \Lp[2](a,b;\C^n)}{\mc{H}x\in H^N(a,b;\C^n)}\\
    \mf{B}x& = W_1R_{ext}\Phi(\mc{H}x), \qquad 
    \mf{B}_dx = W_2R_{ext}\Phi(\mc{H}x)\\
    \mf{C}x& = \tilde{W} R_{ext}\Phi(\mc{H}x), \qquad 
    B_dv = B_d(\cdot)v
  }
  where $R_{ext}$ and $\Phi(\cdot)$ are as in Definition~\ref{def:BPvariables}.
  For these
   definitions 
the properties in Assumption~\ref{ass:BCSass} follow from~\citel{LeGZwa05}{Thm.~4.2} and Lemma~\ref{lem:PHSpassiveCond}. 

\begin{proof}[Proof of Theorem~\textup{\ref{thm:PHSmain}}]
  To apply Theorem~\ref{thm:CLstabpert} we rewrite the feedthrough $D_c>0$ 
  as in~\eqref{eq:DcRewrite}, in which case the boundary control system has the input operator $\mf{B}+D_c \mf{C}$ and the controller~\eqref{eq:controllerintro} has no feedthrough.
  This corresponds to preliminary output feedback $u(t)=-D_c y(t)+\tilde{u}(t)$.
  Denote by $A_{D_c}=\mf{A}\vert_{\ker(\mf{B}+D_c \mf{C})}$ with $\Dom(A_{D_c})=\ker(\mf{B}+D_c \mf{C})$.

  By Lemma~\ref{lem:PHSpassiveCond}, the original system is impedance passive, and since $D_c>0$, the output feedback preserves impedance passivity. 
  The operator  $A_{D_c}$ is dissipative, and straightforward perturbation arguments (similar to those in the proof of Proposition~\ref{prop:CLexistence})  show that $\ran(1-A_{D_c})=X$. 
  Thus $A_{D_c}$ generates a contraction semigroup by the Lumer--Phillips Theorem and
  this semigroup is exponentially stable
  by Assumption~\ref{ass:PHSstab} 
  (with $K=D_c$).
  As shown  in~\citel{HumKur19}{Prop.~II.4}), $C = \mf{C}\vert_{\Dom(A_{D_c})}$ is admissible with respect to the semigroup generated by $A_{D_c}$.

Finally, we need to verify Assumption~\ref{ass:Piwksur}, i.e., that
the transfer function $P_{D_c}(\gl)$ of~\eqref{eq:PHSintro} with feedback $u(t)=-D_c y(t)+\tilde{u}(t)$ satisfies $\re P_{D_c}(\pm i\gw_k)>0$ for all $k$. Define $K_0=\frac{1}{2}D_c>0$ and denote the transfer function of~\eqref{eq:PHSintro}
with output feedback $u(t)=-K_0 y(t)+\tilde{u}(t)$ by
$P_{K_0}(\gl)$.
By Assumption~\ref{ass:PHSstab} $P_{K_0}(\gl)\in \R^{p\times p}$ is well-defined for $\gl\in \set{\pm i\gw_k}_{k=0}^q$, and since~\eqref{eq:PHSintro} has no transmission zeros at $\pm i\gw_k$, $P_{K_0}(\pm i\gw_k)$ are nonsingular for all $k$.
Since $D_c=K_0+K_0$, we have
$P_{D_c}(\pm i\gw_k)
=P_{K_0}(\pm i\gw_k)(I+K_0 P_{K_0}(\pm i\gw_k))\inv
=(P_{K_0}(\pm i\gw_k)\inv+K_0)\inv $ for all $k$. Since $\re (P_{K_0}(\pm i\gw_k)\inv) + K_0>0$, it is easy to show that Assumption~\ref{ass:Piwksur} holds.
The claims now follow from Theorem~\ref{thm:CLstabpert}.
\end{proof}

\section{Application to Atomic Force Microscopy}
\label{sec:example}

As application example we consider the output tracking trajectory problem for a  piezo actuated tube used in positioning systems for Atomic Force Microscopy (see Figure~\ref{fig:AFM} (left)).
\begin{figure}[h!]
\begin{center}
\includegraphics[width=.35\linewidth]{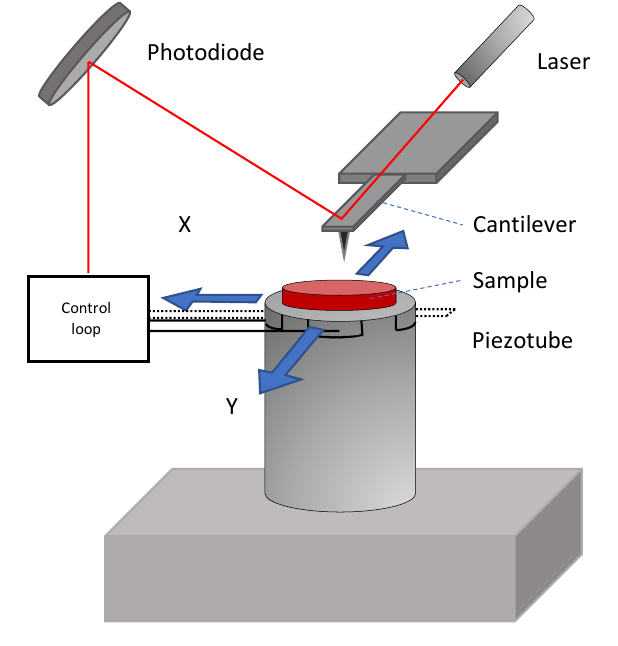}
\hspace{1cm}
 \includegraphics[width=.35\linewidth]{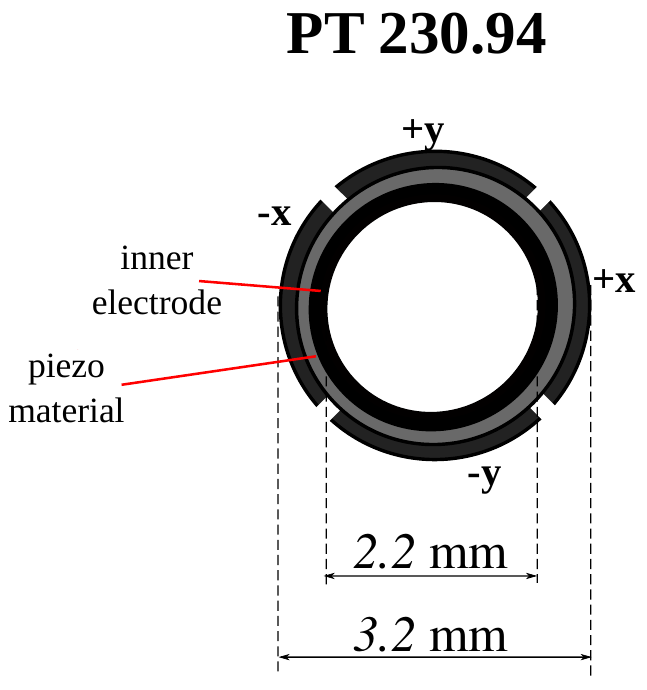}
\caption{Atomic Force Microscopy (left). The piezoelectric tube (right). }
\label{fig:AFM}
\end{center}
\end{figure}

This actuator provides the high positioning resolution and the large bandwidth necessary for the trajectory control during scanning processes. The active part situated at the tip of the flexible tube is composed of three concentric layers: piezo material in between two cylindric electrodes (Figure~\ref{fig:AFM} (right)). The deformation of the active material subject to an external voltage results in an torque applied at the extremity of the tube.   

We consider the motion of the tube in one direction. In this case the structure of the system behaves as a clamped-free beam, represented by the Timoshenko beam model and actuated through boundary control stemming from the piezoelectric action at the tip of the beam. By choosing as state variables the energy variables, namely the shear displacement $x_1(t)=\frac{\partial w}{\partial z}(\cdot,t)-\phi(\cdot,t)$, the transverse momentum distribution $x_2(t)=\rho\frac{\partial w}{\partial t}(\cdot,t)$, the angular displacement $x_3(t)=\frac{\partial \phi}{\partial z}(\cdot,t)$ and the angular momentum distribution $x_4(t)=I_\rho\frac{\partial \phi}{\partial t}(\cdot,t)$
for $t\geq 0$, where $w(z,t)$ is the transverse displacement and $\phi(z,t)$ the rotation angle of the beam, the port-Hamiltonian model of the uncontrolled Timoshenko beam has the form~\eqref{eq:PHSintro1}--\eqref{eq:PHSintro2}
with $\mathcal{H}(\cdot) \equiv \diag \left(K,\frac{1}{\rho},EI, \frac{1}{I_\rho}\right)\in\R^4$,
\begin{equation*}
P_1 = \left[
\begin{array}{c c c c}
0 & 1 & 0 & 0\\
1 & 0 & 0 & 0\\
0 & 0 & 0 & 1\\
0 & 0 & 1 & 0
\end{array}\right],
\qquad
 P_0 = \left[
\begin{array}{c c c c}
0 & 0 & 0 & -1\\
0 & 0 & 0 & 0\\
0 & 0 & 0 & 0\\
1 & 0 & 0 & 0
\end{array}\right]\; 
\end{equation*}
and $G_0=\diag(0,b_w,0,b_\phi)$~\cite{LeGZwa05}.  
Here $\rho$, $I_\rho$, $E$, $I$ and $K$ are the mass per unit length, the angular moment of inertia of a cross section, Young's modulus of elasticity, the moment of inertia of a cross section, and the shear modulus respectively, $b_{w}, b_{\phi}$ the frictious coefficients. 
From Definition \ref{def:BPvariables} considering that $N=1$ and  $Q=P_1$ we get
\[
  \pmat{f_\partial(t) \\ e_\partial(t)} 
=\frac{1}{\sqrt{2}}
\pmat{
  \frac{\partial w}{\partial t}(b)- \frac{\partial w}{\partial t}(a)
  \\
  K\left(\frac{\partial w}{\partial z}(b)-\phi(b)\right)-K\left(\frac{\partial w}{\partial z}(a)-\phi(a)\right)
  \\
  \frac{\partial \phi}{\partial t}(b)-\frac{\partial \phi}{\partial t}(a)
  \\
  EI \frac{\partial \phi}{\partial z}(b)-EI \frac{\partial \phi}{\partial z}(a)
  \\
  K\left(\frac{\partial w}{\partial z}(b)-\phi(b)\right)+K\left(\frac{\partial w}{\partial z}(a)-\phi(a)\right) 
  \\
  \frac{\partial w}{\partial t}(b)+ \frac{\partial w}{\partial t}(a)
  \\
  EI \frac{\partial \phi}{\partial z}(b)+EI \frac{\partial \phi}{\partial z}(a)
  \\
  \frac{\partial \phi}{\partial t}(b)+\frac{\partial \phi}{\partial t}(a)
}
\]

The beam is clamped at point $a$, i.e., $\frac{1}{\rho}x_{2}(a,t)=\frac{1}{I_{\rho}}x_{4}(a,t)=0$ for $t\geq0$ and free/actuated at point $b$, i.e., $Kx_{1}(b,t)=0$ and $EIx_{3}(b,t)=u(t)$ for $ t\geq0$. The angular velocity $\frac{\partial \phi}{\partial t}(b,t)$ at the tip of the beam is measured. The input and output of the system are then of the form~\eqref{eq:PHSintro} with
\eq{
  W_1&=\frac{1}{\sqrt{2}} 
  \pmat{ 0 &0 &0 &1 &0 &0 &1 &0 }\\
  W_2&=\frac{1}{\sqrt{2}}  
  \pmat{
    0 &1 &0 &0 &1 &0 &0 &0 \\
    -1 &0 &0 &0 &0 &1 &0 &0 \\
    0 &0 &-1 &0 &0 &0 &0 &1 
  }\\
  \tilde{W}&= \frac{1}{\sqrt{2}}
  \pmat{
    0 &0 &1 &0 &0 &0 &0 &1
  }
}
%
The matrix $W:=  \pmatsmall{W_1\\ W_2 }$ has full rank and $W \Sigma W^T =0$. Furthermore $\iprod{(W_1^T \tilde{W}+ \tilde{W}^T W_1- \Sigma)g}{g}= 0$ for all $g\in\ker(W_2)$, 
the system is then impedance passive satisfying Assumption~\ref{ass:PHSW}. 
  The system is also exponentially stable and Assumption~\ref{ass:PHSstab} holds.
From Proposition~\ref{prop:CLexistence} the closed loop system has a solution and the regulation error is well defined.

We now build a controller to achieve the robust output tracking for the Piezoelectric tube model.
We use the numerical values given in Table \ref{table1} to achieve a realistic approximation of the dynamics of the piezo actuated tube.

\begin{table}[h!]
\small
\center
\begin{tabular}{|p{2.9cm}|p{3cm}|p{2cm}|l|}
\hline
 Beam's parameters& Value & Simulation parameters & Value \\
\hline
Beam length & 5 cm &$N_f$ & 50 \\
Beam width & 0.3 cm &$a$ & $200$ cm/s \\
Beam thickness &0.2 cm  &$b$ &$100$ cm/s  \\
Material Density & 936  kg/m$^3$
&$c$&0.2 N/m\\
Young's modulus & 4.14 GPa &$\theta$ &0.6\\
Transverse diss.  & $10^{-4}$ N$\cdot$s/m &
 $\gw_1$ & 10 rad/s \\              
coef. & & 
 $\gw_2$ & 15 rad/s\\
Rotational diss.  & $10^{-4}$ N$\cdot$m$\cdot$s/rad 
 &$\gw_3$ &50 rad/s\\
coef. & &
$D_c$ & $0.002$\\ 
&&$\gd_c$ & $0.2$\\ 
\hline
\end{tabular}
\medskip
\caption{\noindent Simulation parameters.\label{table1}}
\end{table}

For the tracking we consider the reference signal
\eq{
  \yref(t) = a\sin(\gw_1 t) + b\cos(\gw_2 t), \qquad a,b\in\R\setminus \set{0}.
}
with two pairs of frequencies $\pm \gw_k$ where $\gw_i>0$, $k\in \left\{1,2\right\}$. 
As an input disturbance signal we consider 50 Hz AC noise coming from the electrical network, hence $w_{dist,2}(t) = c\sin(2\pi50 t + \theta)$ with unknown $c\in\R$ and $\theta\in [0,2\pi]$.
Since the piezo-actuated tube is a single-input single-output system, we can use a controller of the form (with $e(t)=\yref(t)-y(t)$)
 \eq{
   \dot{x}_c(t) &\hspace{-.3ex}=\hspace{-.3ex} \hspace{-.3ex}
\pmat{
0&\gw_1&0&0&0&0\\-\gw_1&0&0&0&0&0\\
0&0&0&\gw_2&0&0\\0&0&-\gw_1&0&0&0\\
0&0&0&0&0&\gw_3\\0&0&0&0&-\gw_3&0\\
}
\hspace{-.3ex}x_c(t) +\hspace{-.3ex} \pmat{\gd_c\\0\\\gd_c\\0\\\gd_c\\0}\hspace{-.3ex}
e(t)
  \\
  u(t)&= \gd_c\pmat{1&0&1&0&1&0}x_c(t)
  + D_c e(t)
}
on $X_c = \R^6$.
%
%
%
%
By Theorem~\ref{thm:PHSmain} 
 the controller achieves asymptotic output tracking of the reference signal $\yref(t)$
if 
 $i\gw_1$, $i\gw_2$, and $i\gw_3$ are not transmission zeros of the system,
 if $D_c>0$,
 and if $\gd_c>0$ is sufficiently small.


\begin{figure}[h!]
\begin{center}
\includegraphics[width=.8\linewidth]{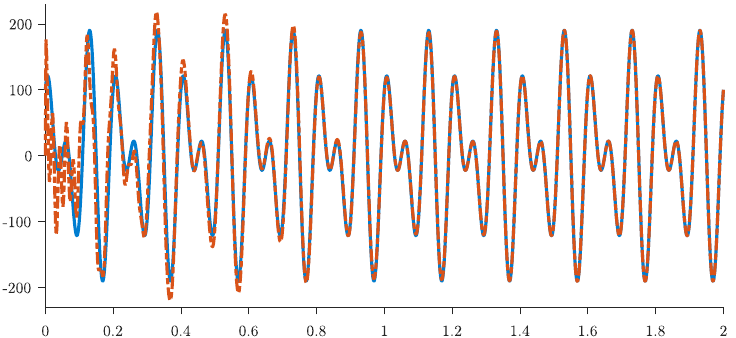}
\caption{Simulation results. 
The controlled output $y(t)$ (dashed red line) and the reference $\yref(t)$ (solid blue line).
}
\label{fig:simu}
\end{center}
\end{figure}


For simulation the Timoshenko beam model was discretized using a structure preserving method based on the Mixed Finite Element Method  \cite{Golo04,Baaiu09JPC}. 
We denote by $N_f$ the number of basis elements, and consequently the full finite dimensional system has order $4N_f$. 
All the numerical values of the parameters related to the simulation can be found in table \ref{table1}. 
Figure \ref{fig:simu} depicts the output tracking performance for the zero initial states of the system and the controller, and exhibits steady convergence of the tracking error to zero.
Due to robustness the output tracking is achieved even if the physical parameters of the piezo actuated tube model 
contain uncertainties or experience changes, as long as the closed-loop system stability is preserved.


\section{Conclusions}
\label{sec:Conclusions}

In this paper we have proposed a constructive method for the design of impedance passive controllers for robust output regulation of port-Hamil\-tonian systems with boundary control and observation. Our results use Lyapunov techniques and extend previous results on this topic by removing the assumption of wellposedness, which is often highly challenging to verify for concrete PDE models.  Future research topics include the design of robust controllers for nonlinear PHS.

\medskip

\noindent\textbf{Acknowledgement.}
The authors are grateful to Jukka-Pekka Humaloja for helpful discussions on 
PHS.

\end{document}